\documentclass[a4paper,11pt]{amsart}

\usepackage[utf8]{inputenc}
\usepackage{amsmath,amsfonts,amsthm}
\usepackage{lmodern}
\usepackage[T1]{fontenc}
\usepackage{xspace}
\usepackage[abbrev]{amsrefs}
\usepackage{tikz}
\usepackage{bm}

\newcommand{\Barany}{B\'ar\'any\xspace}

\DeclareMathOperator{\PPos}{Pos} 
\DeclareMathOperator{\Conv}{Conv} 
\DeclareMathOperator{\Vol}{Vol} 
\DeclareMathOperator{\vol}{vol} 
\DeclareMathOperator{\PProb}{P} 
\DeclareMathOperator{\EE}{E} 
\DeclareMathOperator{\Int}{Int} 
\DeclareMathOperator{\dist}{dist} 

\newcommand{\indicator}[1]{\mathbf{1}_{\{#1\}}} 
\newcommand{\dotp}[2]{#1\cdot #2}
\newcommand{\kk}{\mathbf k} 
\newcommand{\mm}{\mathbf m} 
\newcommand{\nn}{\mathbf n} 
\newcommand{\PPrim}{\mathbb P^d} 
\newcommand{\RR}{\mathbb R} 
\newcommand{\BB}{\mathbb S}

\newcommand{\F}{\mathcal F}
\newcommand{\PP}{\mathcal P}
\newcommand{\C}{\mathcal C}
\newcommand{\T}{\mathcal T} 
\newcommand{\uu}{\mathbf u} 
\newcommand{\vv}{\mathbf v} 
\newcommand{\ww}{\mathbf w} 
\newcommand{\zz}{\mathbf z} 

\newcommand{\xx}{\mathbf x} 
\newcommand{\XX}{\mathbf X} 
\newcommand{\ZZ}{\mathbb Z} 
\newcommand{\ee}{\mathbf e}
\newcommand{\ta}{\mathbf t}
\newcommand{\al}{\alpha}
\newcommand{\la}{\lambda}
\newcommand{\eps}{\varepsilon}

\renewcommand{\aa}{\mathbf a} 
\renewcommand{\epsilon}{\varepsilon} 

\newtheorem{theorem}{Theorem}[section]
\newtheorem{corollary}[theorem]{Corollary}
\newtheorem{lemma}[theorem]{Lemma}
\newtheorem{proposition}[theorem]{Proposition}
\numberwithin{equation}{section}

\author{Imre B\'ar\'any}

\address{Alfr\'ed R\'enyi Institute of Mathematics,
Hungarian Academy of Sciences, 13-15 Re\'altanoda Street
1053 Budapest, Hungary, and
Department of Mathematics,
University College London,
Gower Street, London WC1E 6BT, UK}

\email{barany.imre@renyi.mta.hu}

\author{Julien Bureaux}
\address{Universit\'e Paris Ouest Nanterre La D\'efense,
200 avenue de la R\'epublique,
92000 Nanterre,
France}
\email{julien.bureaux@math.cnrs.fr}

\author{Ben Lund}
\address{
 Department of Mathematics,
Boyd Graduate Studies Research Center,
University of Georgia,
Athens, GA~30602, USA
}
\email{lund.ben@gmail.com}

\begin{document}

\title{Convex cones, integral zonotopes, limit shape}

\begin{abstract} Given a convex cone $C$ in $\RR^d$, an integral zonotope
$T$ is the sum of segments $[0,\vv_i]$ ($i=1,\ldots,m$) where each $\vv_i\in
C$ is a vector with integer coordinates. The endpoint of $T$ is
$\kk=\sum_1^m \vv_i$. Let $\T(C,\kk)$ be the family of all integral zonotopes in $C$ whose endpoint
is $\kk \in C$. We prove that, for large $\kk$, the zonotopes in
$\T(C,\kk)$ have a limit shape, meaning that, after suitable scaling, the
overwhelming majority of the zonotopes in $\T(C,\kk)$ are very close to a
fixed convex set. We also establish several  combinatorial properties of a
typical zonotope in $\T(C,\kk)$.
\end{abstract}

\maketitle

\section{Introduction and main results}

This paper is about convex cones $C$ in $\RR^d$, integral zonotopes
contained in $C$, and their limit shape.  The cone $C$ is going to
be closed, convex and pointed (that is no line lies in $C$) and its
interior, $\Int C$, is non-empty.  We write $\C$ or $\C^d$ for the set
of these cones.

A convex (lattice) polytope $T \subset C$ is an integral zonotope if
there exists $m \in \mathbb N$ and $\vv_1,\dots,\vv_m \in \ZZ^d \cap C$
(that is, each $\vv_i$ is lattice point in $C$) such that
\begin{align*}
  T&=\left\{\sum_{i=1}^m \al_i\vv_i \mid (\alpha_1,\dots,\alpha_m) \in [0,1]^m\right\}\\
  &=\Conv \left \{\sum_{i=1}^m \epsilon_i\vv_i \mid (\epsilon_1,\dots,\epsilon_m) \in \{0,1\}^m\right\},
\end{align*}

The multiset $V=\{\vv_1,\ldots,\vv_m\}\subset \ZZ^d$ determines $T=T(V)$
uniquely, of course, but not conversely.  More about this later. The
endpoint of $T$ is just $\sum_{i=1}^m\vv_i$.  Define $\T(C,\kk)$ as
the family of all integral zonotopes in $C$ whose endpoint is $\kk \in
\ZZ^d\cap \Int C$. Clearly, $\T(C,\kk)$ is a finite set. Let $p(C,\kk)$
denote its cardinality.

The main result of this paper is that, for large $\kk$, the overwhelming majority of the elements of $\T(C,\kk)$ are very close to a fixed convex set $T_0=T_0(C,\kk)$ which is actually a zonoid. We write $\dist(A,B)$ for the Hausdorff distance of the sets $A,B \subset \RR^d$. Here comes our main result.

\begin{theorem}\label{th:lshape} Given $C \in \C^d$ ($d\ge 2$) and $\kk \in \Int C$ there is a convex set $T_0=T_0(C,\kk)$ such that for every $\epsilon>0$,
\[
\lim_{n \to \infty}\frac{\operatorname{card}\left\{T \in \T(C,n\kk) \mid \dist (\frac 1n T,T_0)>\epsilon\right\}}{p(C,n\kk)}=0.
\]
\end{theorem}

This result has been known for $d=2$. Twenty years ago, \Barany{}~\cite{barany_limit_1995}, Sinai{}~\cite{sinai_probabilistic_1994} and Vershik~\cite{vershik_limit_1994} proved the existence of a limit shape for the set of all convex lattice polygons lying in the square $[-n,n]^2$ endowed with the uniform distribution. Although not all convex lattice polygons are (translates of) zonotopes, case $d=2$ of Theorem~\ref{th:lshape} follows directly from their result.
The approach of these papers relies on a natural link between convex lattice polygons on the first hand, and integer partitions on the other hand.

In addition to Theorem~\ref{th:lshape}, the asymptotic behaviour as $n\to \infty$ of $p(C,n\kk)$ can also be determined.

\begin{theorem}\label{th:size} Under the above conditions on $C$ and $\kk$ there is a number $q(C,\kk)>0$ such that, as $n$ tends to infinity,
\[
n^{-\frac d{d+1}}\log p(C,n\kk) \longrightarrow c_d\, q(C,\kk),
\]
where $c_d = \sqrt[d+1]{\frac{\zeta(d+1)}{\zeta(d)}(d+1)!}$ depends only on the dimension.
\end{theorem}

As we shall see in Section 3, $q(C,\kk)$ is a constant multiple of the $d+1$st root of the volume of the minimal cap of $C$ containing $\kk$.

The next section connects integral zonotopes and strict integer partitions. Section~\ref{sec:limiting_zonoid} is about the limiting zonoid and some examples. The basic probabilistic model and proofs of the main results are in Section~\ref{ssec:model}. Section~\ref{sec:number_vertices} establishes several  combinatorial properties of a typical zonotope $T$ in $\T(C,\kk)$, namely we estimate the number of $i$-dimensional faces of $T$, for $i=0,1,\ldots,d-1$.
Section~\ref{sec:convex_bodies} includes proofs of the existence of limit shapes for integral zonotopes chosen in arbitrary convex bodies in $\mathbb{R}^2$, and in hypercubes for higher dimensions.

\section{Strict integer partitions}
\label{sec:strict_partitions}

A multiset $V=\{\vv_1,\ldots,\vv_m\} \subset \ZZ^d\cap C$ determines the zonotope $T=T(V)$ uniquely but, as remarked earlier, $T$ does not determine $V$ uniquely. We are going to choose a suitable multiset $W \subset \ZZ^d \cap C$ uniquely. This is fairly simple. First let $\PPrim$ denote the primitive vectors in $\ZZ^d$; a vector $\zz=(z_1,\ldots,z_d) \in \ZZ^d$ is primitive if $\gcd(z_1,\ldots,z_d)=1$. Note that $0\notin \PPrim$. Given $T=T(V)$ with generators $V=\{\vv_1,\ldots,\vv_m\}\subset \ZZ^d\cap C$, there is a unique multiset $W=\{\ww_1,\ldots,\ww_{\ell}\}\subset \PPrim\cap C$ that generates the same zonotope, that is, $T(V)=T(W)$. Indeed, each $\vv_i$ can be written uniquely as $h\ww$ for some $\ww \in \PPrim$ and $h \in \ZZ_+$. Then put $h$ copies of $\ww$ in $W$. This way we get a multiset $W \subset \PPrim \cap C$ such that   $T=T(W)$. It is easy to check (we omit the details) that if $T=T(U)$ for some other multiset $U \subset \ZZ^d\cap C$, then the above construction gives the same $W$. That means that $W$ is uniquely determined by $T$.

We have just defined a one-to-one correspondence between lattice zonotopes $T$ lying in $C$ and multisets of primitive vectors $W \subset \PPrim \cap C$. If the endpoint of $T$ is $\kk$, then
\[
\sum_{i=1}^m \ww_i=\kk.
\]

Such a multiset is called a \emph{strict integer partition} of the vector $\kk \in \ZZ^d$ from the cone $C$. It can be alternatively described by the family $(\omega(\xx))_{\xx \in \PPrim}$ of multiplicities $\omega(\xx) = \operatorname{card}\left\{ j \in \{1,\dots,\ell\} \mid \ww_j = \xx\right\}$ of the available parts $\xx \in \PPrim \cap C$. Notice that for any partition, there is only a finite number of vectors $\xx \in \PPrim \cap C$ such that $\omega(\xx) \neq 0$. Therefore, picking a strict partition is actually equivalent to picking a function $\omega : \PPrim \cap C \to \ZZ_+$ with finite support. For any such function $\omega$ we define
\[
    \XX(\omega) := \sum_{\xx \in \PPrim \cap C} \omega(\xx) \,\xx.
\]
With this notation, the fact that $\omega$ describes a partition of $\kk$ corresponds to the condition $\XX(\omega) = \kk$.

\bigskip
One can consider non-strict partitions as well, that is, the uniform distribution on all multisets $V=\{\vv_1,\ldots,\vv_m\} \subset \ZZ^d\cap C$ with $\sum_1^m \vv_i=\kk$ (where $\vv_i\ne 0$), so the same zonotope may appear several times. The results of this paper remain valid in this case as well but some constants are different. For instance, Theorem~\ref{th:size} remains valid except that the constant $c_d$ is slightly different, namely, no division by $\zeta(d)$ is required. We omit the details.

\section{The limiting zonoid}
\label{sec:limiting_zonoid}

The dual $C^o$ of a cone $C \in \C^d$ is defined, as usual, via
\[
C^o=\{ \uu \in \RR^d \mid \forall \xx\ \in C \setminus \{0\},\; \uu \cdot\xx >0 \}.
\]
Note that the dual $C^o$ is an open cone, which is convenient for our purposes. Its closure is in $\C^d$ as one can see easily. Given $\uu\in C^o$ and $t>0$ we define the corresponding {\sl section} $C(\uu=t)$ and {\sl cap} $C(\uu\le t)$ of $C$ by
\begin{eqnarray*}
C(\uu=t)&=&\{\xx\in C \mid \uu\cdot\xx=t\},\\
C(\uu\le t)&=&\{\xx\in C \mid \uu\cdot\xx\le t\}.
\end{eqnarray*}

We are going to use the following result of Gigena~\cite{gigena_integral_1978}, see also \cite{faraut_analysis_1994}.

\begin{theorem}\label{th:gigena} Given $C \in \C^d$ and $\aa \in \Int C$ there is a unique $\uu=\uu(C,\aa)$ such that
\begin{itemize}
\item  $\aa$ is the center of gravity of the section $C(\uu=1)$,
\item  $C(\uu\le 1)$ has minimal volume among all caps of $C$ that contain $\aa$, moreover, $C(\uu\le 1)$ is the unique cap with this property.
\end{itemize}
\end{theorem}

Clearly $\uu \in C^{\circ}$.  It follows that $\frac d{d+1}\aa$ is the center of gravity of the cap $C(\uu\le 1)$:
\[
\frac d{d+1}\aa= \frac 1{\Vol C(\uu\le 1)}\int_{C(\uu\le 1)} \xx\, d\xx.
\]
Further, there is a $\la >0$ such that
\[
\aa= \int_{C(\uu\le \la)} \xx\, d\xx.
\]
One can check directly that this $\la $ is unique and is given by
\[
\lambda= \left(\frac {d+1}d \frac 1{\Vol C(\uu\le 1)}\right)^{1/(d+1)}.
\]
Define
$Q=Q(C,\aa)=C(\uu\le \la)$ with this $\la$. An easy computation shows that
\begin{eqnarray*}
q(C,\aa)&:=&\Vol Q(C,\aa)=\la^d \Vol C(\uu\le 1)\\
&=& \left(\left(1+\frac 1d\right)^d\Vol C(\uu\le 1)\right)^{1/(d+1)}.
\end{eqnarray*}

We will show later in Section~\ref{ssec:number_zonotopes}, that this is the $q(C,\kk)$ appearing in Theorem~\ref{th:size}.

\medskip
{\bf Example 1.} Let $C=\RR^d_+=\PPos \{\ee_1,\ldots,\ee_d\}$ be the positive orthant of $\RR^d$ where the $\ee_i$ form the standard basis of $\RR^d$, and $\aa=(a_1,\ldots,a_d) \in \Int C$. Then the section $C(\uu=1)$ is the intersection of $C$ with the hyperplane passing through the points $da_i\ee_i$, ($i=1,\ldots ,d$), and $\Vol C(\uu \le 1)=\frac 1{d!}d^da_1\dots a_d$. Consequently
\[
q(C,\aa)=\left( \frac {(d+1)^d}{d!} a_1\ldots a_d\right)^{1/(d+1)}.
\]

\medskip
{\bf Example 2.} 
 Let $C$ be the circular cone in $\RR^3$ of equation $x^2 + y^2 \leq z^2$ with $z \geq 0$ and consider $\kk=(0,0,1)$.
 The minimal cap of $C$ containing $\kk$ is the one cut off by the plane $z=1$. Its volume is $\pi/3$, so in this case
 \[
 q(C,\kk)=\left(\frac {4^3\pi}{3^4}\right)^{1/4} =1.255294...
 \]

Next we explain what the limiting zonoid $T_0=T_0(C,\kk)$ from Theorem~\ref{th:lshape} is. Its support function is given by
\[
h_{T_0}(\vv)=\int \vv \cdot\xx\, d\xx
\]
where the integral is taken over $Q(C,\kk)\cap C(\vv  \ge 0)$. Similarly, the point $\ta(\vv)$ where the hyperplane orthogonal to $\vv$ supports $T_0(C,\kk)$ is $\int \xx d\xx$ with the integral taken over the same set as above. So the boundary point of $T_0$ with outer normal $\vv$ is given by
\begin{equation}\label{eq:argmax}
\ta(\vv)=\int_{Q(C,\kk)\cap C(\vv  \ge 0)} \xx \,d \xx.
\end{equation}
As expected, $\ta(\vv)=\kk$ for $\vv \in C^{\circ}$ and $\ta(\vv)=\mathbf 0$ for $\vv \in -C^{\circ}$. The proof of these facts follow from the proof of Propositions~\ref{thm:zono_shape} and \ref{thm:zono_shape2}.

{\bf Remark.} When computing $\ta(\vv)$ we can integrate over $C(\uu\le 1)\cap C(\vv\ge 0)$ instead of $Q(C,\kk)\cap C(\vv  \ge 0)$ and use a homothety (with centre the origin) so that $\ta(\vv)=\kk$ for $\vv \in C^{\circ}$.

\medskip
Thus for instance in {\bf Example 1} one can, in principle, determine the limiting zonotope. Here $C=\RR^d_+$ and we may choose $\kk$ to be the all one vector $\mathbf 1$ since the whole question is linearly invariant (or equivariant if you wish). Write $\triangle$ for the convex hull of the vectors $d\ee_i$, $i=1,\ldots,d$ and the origin. Clearly $\triangle=C(\uu\le 1)$. The boundary points of the limiting zonotope $T_0(\RR^d_+,\mathbf 1)$ are given by
\[
\int_{\triangle \cap C(\vv\ge 0)}\xx \,d\xx
\]
scaled properly as explained in the Remark above.

The computation is easy when $d=2$. Then $\vv=(t,-s)$ with $s,t\ge 0$ and $s+t=1$, say. Then $\triangle \cap C(\vv\ge 0)$ is a triangle with vertices $(0,0),(1,0),(s,t)$. The integral in question is equal to the area of the triangle (which is $t/2$) times its centre of gravity (which equals $(1+s,t)/3$). Using $s=1-t$ and suitable scaling we get $\ta(\vv)=(2t-t^2,t^2)$. With $x,y$ coordinates this is just $x+y=2\sqrt y$ (for $x\ge y$), a parabola arc. Thus the boundary of $T_0$ consists of two parabola arcs given by the equations $x+y=2\sqrt y$ (for $x\le y$) and  $x+y=2\sqrt y$ (for $x\ge y$), which is the same as the limit shape in $\RR^2$ of convex lattice polygons, see  \Barany{}~\cite{barany_limit_1995} and  Vershik~\cite{vershik_limit_1994}.

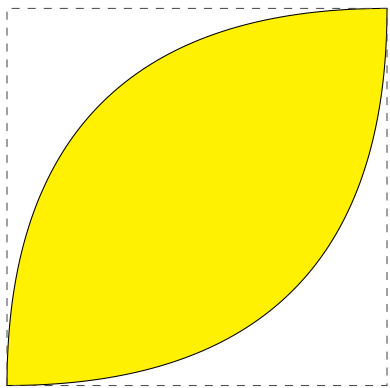
\begin{figure}[h]
  \begin{center}
    \begin{tikzpicture}[scale=5]
      \filldraw [fill=yellow] plot [samples=100, domain=0:1] (2*\x-\x^2, \x^2) -- plot [samples=100, domain=0:1] (1-2*\x+\x^2, 1-\x^2) -- cycle;
      \draw [color=black!70,dashed] (0,0) -- (0,1) -- (1,1) -- (1,0) -- (0,0);
    \end{tikzpicture}
    \caption{The limiting zonoid in dimension 2.}
  \end{center}
\end{figure}

The same method works in $\RR^3$. This time the previous triangle is replaced by the simplex with vertices $(0,0,0),(1,0,0),(s,t,0),(u,0,v)$ with $s,t,u,v\ge 0$ and $s+t=1$ $u+v=1$. The integral in question is the volume of this simplex ($tv/6$) times its centre of gravity ($(1+u+s,t,v)/4$). Using $s=1-t$ and $u=1-v$ again we get $\ta(\vv)=(3tv-t^2v-tv^2,t^2v,tv^2)$ . Thus the equation of the boundary of $T_0$ is $x+y+z=3\sqrt[3]{yz}$; this holds when $x-2y+z\ge 0$ and $x+y-2z \ge 0$, as one can check directly. $T_0$ is centrally symmetric with respect to center $(1/2,1/2,1/2)$. Its boundary is made up of six pieces that come in symmetric pairs. The piece in the region determined by inequalities  $x-2y+z\ge 0$ and $x+y-2z\ge 0$ is given by the equation $x+y+z=3 \sqrt[3]{yz}$. Other pieces are given by equations $x+y+z=3 \sqrt[3]{xz}$ and $x+y+z=3 \sqrt[3]{xy}$, and the reflections with respect to the center.

\begin{figure}[h]
\begin{center}
    \includegraphics[scale=0.5]{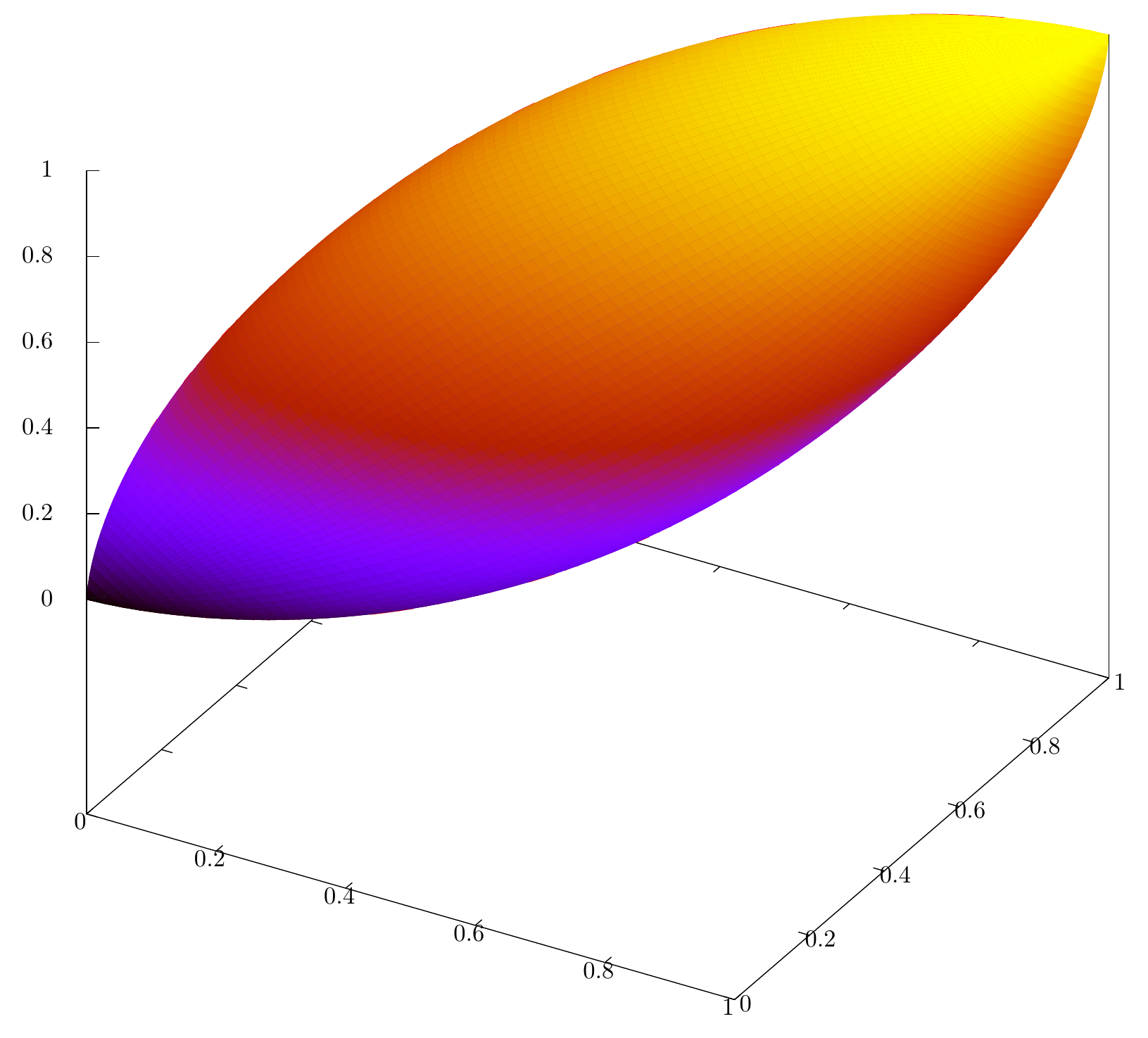}
\caption{The limiting zonoid in dimension 3.}
\end{center}
\end{figure}

The same method works in higher dimensions. There $\triangle \cap C(\vv\ge 0)$ is not a single simplex, one has to triangulate it into simplices, and on each simplex, the integral in question can be computed the same way as above. We have not carried out this computation. Yet one can show that a certain part of the boundary of $T_0$ in $\RR^d$ is described by the equation $x_1+\ldots+x_d=d(x_2\ldots x_d)^{1/d}$.

\medskip
The same limit shape comes up in another way as well. We are going to explain, rather informally, how this happens. We start with a simple proposition.

\begin{proposition}\label{prop:max} Assume $S \subset C$ is a closed set and $\int_S\xx \,d\xx=\aa$. Then $\Vol S \le q(C,\aa)$ and equality holds iff $S=Q(C,\kk)$.
\end{proposition}

\begin{proof} It is clear that $\int_S\xx \,d\xx=\aa$ implies that $\int_S\uu \cdot\xx\, d\xx=\uu\cdot \aa$. Note that $\int_Q\uu \cdot\xx \,d\xx=\uu\cdot \aa$ as well. Then
\[
\int_{S \setminus Q}\uu \cdot\xx \,d\xx =\int_{Q \setminus S}\uu \cdot\xx \,d\xx.
\]
The value of $\uu \cdot\xx$ is larger on $S \setminus Q$ than on $Q \setminus S$. This shows that $\Vol S \setminus Q$ is smaller than $\Vol (Q \setminus S)$ unless both are equal to zero.
\end{proof}

\medskip
Assume that $\kk \in \ZZ^d \cap \Int C$ and consider the cap $Q(C,n\kk)$ when $n$ is large. Define $V_n=\ZZ^d \cap Q(C,n\kk)$. It is known (and follows from an  estimate similar to the one in Lemma~\ref{lem:cubature}) that, as $n$ goes to infinity,
\[
\sum_{\xx \in  V_n}\xx=n\kk(1+o(1)).
\]
Suppose now that $\sum_{\xx \in S}\xx=n\kk$ for some $S \subset \ZZ^d \cap C$. Under these conditions Proposition~\ref{prop:max} implies that $|S| \le |V_n|(1+o(1))$.

This shows that the (asymptotically) largest set $S \subset \ZZ^d \cap C$ with $\sum_{\xx \in S}\xx=n\kk$ is $V_n$. The integral zonotope $T(V_n)$ has its endpoint close to $n\kk$. The limiting zonoid $T_0(C,\kk)$ from Theorem~\ref{th:lshape} turns out to be the limit of the zonotopes $\frac 1n T(V_n)$ as $n\to \infty$.

\section{The probabilistic approach}
\label{ssec:model}

Throughout this paper we work under the following condition:
\begin{equation}\label{eq:fix}
C \in \C^d \mbox{ and } \kk \in \Int C \cap \ZZ^d \mbox{ are fixed.}
\end{equation}

In this section, we give a proof of Theorem~\ref{th:lshape} and Theorem~\ref{th:size} based on statistical mechanics techniques.
Obviously, the statement of the theorem itself is already of probabilistic nature.

\subsection{Description of the model}

Let $\Omega$ be the set of all functions from $\PPrim \cap C$ to $\ZZ_+$ with finite support.
Recall from Section~\ref{sec:strict_partitions} the correspondence between lattice zonotopes, strict integer partitions, and multiplicity functions $\omega \in \Omega$.
In particular, each function $\omega$ is associated to a unique zonotope $T(\omega)$, and conversely. Moreover, the endpoint of $T(\omega)$ is
\[
    \XX(\omega) = \sum_{\xx \in \PPrim \cap C} \omega(\xx)\, \xx.
\]
The probability distribution $\mathrm Q_n$ on $\Omega$, which is defined for all $\omega \in \Omega$ by
\[\mathrm Q_n(\omega) = \dfrac{1}{p(C,n\kk)}\mathbf{1}_{\{\XX(\omega) = n\kk\}}
\] is exactly the uniform distribution on $\T(C,n\kk)$.
The conclusion of Theorem~\ref{th:lshape} can be stated as convergence in probability for this distribution:
\[
	\dist(\tfrac1n T, T_0) \xrightarrow[n\to\infty]{\mathrm Q_n} 0.
\]

We now define a \emph{new} probability distribution $\PProb_n$ on $\Omega$, which will turn out to behave roughly like $\mathrm Q_n$ when $n$ tends to $+\infty$. It depends on two parameters $\uu$ and $\beta_n$ fixed throughout the paper as follows:
\begin{equation}\label{eq:param}
\uu=\uu(C,\lambda\kk),\qquad \beta_n = \sqrt[d+1]{\frac{\zeta(d+1)}{\zeta(d)}\frac1n},
\end{equation}
where $\uu(C,\lambda\kk) = \lambda^{-1}\,\uu(C,\kk) \in C^{\circ}$ is defined by Theorem~\ref{th:gigena}, with $\lambda > 0$ chosen such that
\[
(d+1)! \int_{C(\uu \leq 1)} \xx \,d\xx = \kk.
\]
The reason for these choices will become apparent in Proposition~\ref{prop:central_limit}.
The probability distribution $\PProb_n$ is then defined for all $\omega\in\Omega$ by
\[
    \PProb_n(\omega) = \frac{1}{Z_n}e^{-\beta_n\dotp{\uu }{ \XX(\omega)}}, \qquad \text{where } Z_n = \sum_{\omega \in \Omega} e^{- \beta_n \dotp{\uu}{ \XX(\omega)}}.
\]
This definition, which is directly inspired by statistical mechanics, is a special case of the Boltzmann distribution. It is a generalization of the model introduced by Sinai{} \cite{sinai_probabilistic_1994} for convex lattice polygonal lines.
Here $Z_n$ is the so-called \emph{partition function} of the model. The sum defining $Z_n$ is easily seen to be convergent.

The crucial observation is the following: since $\PProb_n(\omega)$ only depends on the value $\XX(\omega)$, we see that for all $\mm \in \ZZ^d$,
\begin{equation}
    \label{eq:proba}
    \PProb_n[\XX = \mm] = \frac{p(C,\mm)}{Z_n}e^{-\beta_n\dotp{\uu}{\mm}}.
\end{equation}
In particular, the distribution induced on $\Omega$ by $\PProb_n$ conditional on the event $\{\XX = n\kk\}$ is exactly $\mathrm Q_n$. For all events $E \subset \Omega$,
\[
	\mathrm Q_n[E] = \PProb_n[E \mid \XX = n\kk] = \frac{\PProb_n[E \cap \{\XX = n\kk\}]}{\PProb_n[\XX = n\kk]}.
\]
Consequently, our strategy of proof for Theorem~\ref{th:lshape} is to establish first a strong limit shape result for the distribution $\PProb_n$ and then use this relation with $E = \{\dist(\frac1n T,T_0) > \epsilon\}$ to get the result for $\mathrm Q_n$. This deduction will only be possible if we show that $\PProb_n[\XX = n\kk]$ is not too small.

In comparison with $\mathrm Q_n$, the distribution $\PProb_n$ has a much simpler probabilistic structure. It is a product distribution on $\Omega$ since, by definition of $\XX$,
\[
	e^{-\beta_n \dotp{\uu}{\XX(\omega)}} = \prod_{\xx \in \PPrim\cap C} e^{-\beta_n \omega(\xx)\dotp{\uu}{\xx}}.
\]	
Thus, the family of random variables $(\omega(\xx))_{\xx \in \PPrim \cap C}$ is independent under $\PProb_n$.
In addition, one can see that for all $\xx \in \PPrim \cap C$, the integer-valued random variable $\omega(\xx)$ has geometric distribution with parameter $e^{-\beta_n\dotp{\uu}{ \xx}}$ (failure probability), that is to say
\begin{equation}\label{eq:geomdistr}
    \PProb_n[\omega(\xx) = i] = (1-e^{-\beta_n \dotp{\uu}{\xx}})\,e^{- i \beta_n \dotp{\uu}{\xx}}, \qquad i \in \ZZ_+.
\end{equation}
These two facts lead to a remarkable factorization of $Z_n$ as convergent product which is similar to Euler's formula,
\[
    Z_n = \prod_{\xx \in \PPrim \cap C} \frac{1}{1-e^{-\beta_n \dotp{\uu}{\xx}}}.
\]
Taking logarithms, we obtain finally a series expansion for $\log Z_n$,
\begin{equation}
  \label{eqn:logZ}
    \log Z_n = \sum_{\xx\in\PPrim \cap C} \sum_{r \geq 1} \frac{e^{-r \beta_n\dotp{\uu}{\xx}}}{r}.
\end{equation}
The same expression is used in \cite{barany_limit_1995} and \cite{sinai_probabilistic_1994}.

\subsection{Asymptotic behaviour after rescaling}

In this section, we investigate the asymptotic distribution of the random vector $\XX$. It turns out that, after a proper rescaling, it can described completely by the Laplace transform (also called a characteristic function in \cite{faraut_analysis_1994}) of the cone $C$, which is defined by
\[
	\Lambda_C(\vv) = \int_C e^{-\dotp{\vv}{\xx}}\,d\xx, \qquad \vv \in C^\circ.
\]
When there is no risk of confusion with another cone, we write $\Lambda = \Lambda_C$.

\begin{proposition}
    \label{prop:central_limit}
    Let $\bm\mu_n$ and $\Gamma_n$ denote respectively the mean value and the covariance matrix of the random vector $\XX$ under the distribution $\PProb_n$. Then
    \begin{equation}
        \lim_{n\to+\infty} \frac{1}{n}\, \bm\mu_n = - \nabla \Lambda(\uu) = \kk, \quad\text{and}\quad
        \lim_{n\to+\infty} n^{-\frac{d+2}{d+1}}\,\Gamma_n = \nabla^2 \Lambda(\uu),
    \end{equation}
    where $\nabla \Lambda$ and $\nabla^2 \Lambda$ denote respectively the gradient and the Hessian matrix.

    Moreover, the distribution of $\XX$ satisfies a central limit theorem in the sense that
    \[
        \Gamma_n^{-1/2}(\XX - \bm\mu_n) \xrightarrow[n\to+\infty]{\text{law}} \mathcal{N}(0,I_d).
    \]
\end{proposition}

\begin{proof}
    We introduce a function $Z : C^\circ \to \RR_+$ such that $Z_n = Z(\beta_n \uu)$. It is defined for all $\vv\in C^\circ$ by
    \[
        Z(\vv) = \sum_{\omega\in\Omega} e^{-\vv \cdot \XX(\omega)}.
    \]

    Let $\beta = \beta_n$ in the proof.
    We start with the statement about $\bm\mu_n$. 
    It is a well known fact that in such an exponential model, the first moments of $\XX$ are given by the logarithmic derivatives of the function $Z$. In the case of the mean value, it results from the following simple computation:
    \[
	    \bm\mu_n = \EE_n[\XX] = \sum_{\omega \in \Omega} \frac{e^{-\beta \dotp{\uu}{ \XX(\omega)}}}{Z(\beta\uu)}  \XX(\omega) =  -\frac{\nabla Z(\beta\uu)}{Z(\beta\uu)} =      -\nabla \log Z(\beta \uu).
    \]
    Now, a straightforward generalization of formula \eqref{eqn:logZ} to the function $Z$ leads to a series expansion of $\nabla \log Z$, namely
    \[
        -\nabla \log Z(\beta \uu) = \sum_{r \geq 1} \sum_{\xx \in \PPrim \cap C} e^{-r\beta \dotp{\uu}{\xx}} \xx.
    \]
    For every index $r$ less than $1/\beta$, we approximate the
    summation on $\PPrim\cap C$ with the corresponding $d$-dimensional integral
    by using Proposition~\ref{prop:primitive_sums}.
    Since the terms with $r > 1/\beta$ contribute only $O(1/\beta)$,
    we obtain after summation
    \[
	    \bm\mu_n = \frac{1}{\beta^{d+1}} \frac{\zeta(d+1)}{\zeta(d)}\int_C e^{-\dotp{\uu}{\xx}}\xx\,d\xx + O\left(\frac{1}{\beta^d}\right)
    \]
    in the limit $\beta\to 0$.  The integral here is obviously $-\nabla \Lambda(\uu)$. We need check that it is also equal to $\kk$. This is done by applying the fundamental theorem of calculus and the Fubini-Tonelli theorem:
    \[
      \int_C e^{-\uu\cdot \xx} \xx\,d\xx
        = \int_C \int_{\uu\cdot \xx}^\infty e^{-t}\,dt\,\xx\,d\xx
        = \int_0^\infty e^{-t}\int_{C(\uu \leq t)} \xx\,d\xx \,dt.
    \]
    An homothetic change of variable yields therefore by homogeneity:
    \[
      \int_C e^{-\uu\cdot \xx} \xx\,d\xx =
      \left[\int_0^\infty t^{d+1}e^{-t}\,dt\right]\int_{C(\uu\leq 1)} \xx\,d\xx =
      (d+1)! \int_{C(\uu\leq 1)} \xx\,d\xx = \kk.
    \]
    So the first part of the proposition is finally proven by replacing $\beta$ with its expression as a function of $n$.

    \medskip
    The proof of the asymptotic behaviour of the covariance matrix
    $\Gamma_n$ is entirely similar except that we consider the second
    order derivatives of the logarithmic partition function. We omit the details.

    \medskip
    We now turn to the central limit result.
    Recall that the random variables $\omega(\xx)$ are independent and each $\omega(\xx)$
    follows the geometric distribution with parameter $e^{-\beta \dotp{\uu}{x}}$.
    In particular, $\Gamma_n^{-1/2}(\XX-\bm\mu_n)$ is the
    sum of the independent random variables $(\omega(\xx) -
    \EE[\omega(\xx)])\,\Gamma_n^{-1/2}\xx$
    of mean value $0$.
    We are going to check the classical Lyapunov condition on third moments, which implies
    the central limit theorem for independent, but not necessarily
    identically distributed, summands. We must show that
    \[
	L_n = \sum_{\xx \in \PPrim \cap C} \EE\left[\left\|(\omega(\xx) - \EE[\omega(\xx)])\,\Gamma_n^{-1/2}\xx \right\|^3\right]
    \]
    goes to $0$ as $n$ tends to $+\infty$.
    Elementary computations on the geometric distribution yield,
    for all $\xx \in \PPrim\cap C$,
    \[
	    \EE[\omega(\xx)] = \frac{e^{-\beta\dotp{\uu}{\xx}}}{1-e^{-\beta\dotp{\uu}{\xx}}}, \quad\text{and}\quad
        \EE\left[\left|\omega(\xx) - \EE[\omega(\xx)]\right|^3\right] \leq \frac{3\,e^{-\beta \dotp{\uu}{\xx}}}{(1-e^{-\beta \dotp{\uu}{\xx}})^3}.
    \]
    Letting $\|\Gamma_n^{-1/2}\|$ denote the operator norm, we obtain therefore,
    \[
        L_n \le 3\|\Gamma_n^{-1/2}\|^3 \sum_{\xx \in \PPrim \cap C} \frac{\|\xx\|^3 e^{-\beta \dotp{\uu}{\xx}}}{(1-e^{-\beta \dotp{\uu}{\xx}})^3}.
    \]
    From the first part of the proposition, we already know that
    $\|\Gamma_n^{-1/2}\|$ is of order ${n^{-\frac 1 2(d+2)/(d+1)}}$.
Moreover, Proposition~\ref{prop:primitive_sums} shows that
\[
	\sum_{\xx \in \PPrim \cap C} \frac{\|\xx\|^3 e^{-\beta \dotp{\uu}{\xx}}}{(1-e^{-\beta \dotp{\uu}{\xx}})^3} \quad\text{and}\quad n^{\frac{d+3}{d+1}}\int_C \|\xx\|^3 e^{-\dotp{\uu}{\xx}}d\xx
\]
are of the same order of magnitude as $n$ tends to $+\infty$.
Accordingly, $L_n$ is at most of order $n^{-d/(2d+2)}$, hence tends to $0$ and we can indeed apply Lyapunov's central limit theorem. The proof is complete.
\end{proof}

Letting $\sigma_n = n^{\frac{d+2}{2d+2}}$, the central limit theorem of Proposition~\ref{prop:central_limit} can now be stated as
\begin{equation}
    \label{eq:central_limit}
    \frac{\XX - n\kk}{\sigma_n} \xrightarrow[n\to\infty]{\text{law}} \mathcal{N}(0,\nabla^2 \Lambda(\uu)).
\end{equation}
The proof of the next proposition will use  this fact.

\begin{proposition}[Weak local limit estimate]
    \label{prop:zono_local}
    \[
        \PProb_n[\XX = n\kk] = e^{-o(n^{d/(d+1)})}
    \]
\end{proposition}

\begin{proof}
    We begin with the following simple remark: if $\mm \in \ZZ^d \cap C$ and $\nn \in \ZZ^d \cap C$ satisfy $\nn - \mm \in C$, then $p(C,\nn) \geq p(C,\mm)$. This explains the introduction of the following set,
    \[
        A_n = \{\mm \in \ZZ^d \cap C \mid n\kk - \mm \in C \text{ and } \dotp{ \uu}{(n\kk - \mm) } \leq \sigma_n\}.
    \]
    As a consequence of \eqref{eq:central_limit}, there exists $c > 0$ such that $\PProb_n[\XX \in A_n] \geq c$ for all $n$ large enough. In addition, equation~\eqref{eq:proba} and the remark above imply that
    \begin{align*}
        \PProb_n[\XX \in A_n] &= \sum_{\mm \in A_n} p(C,\mm)e^{-\beta_n\dotp{\uu}{\mm}}\\
		&= \sum_{\mm \in A_n} \frac{p(C,\mm)}{p(C,n\kk)} e^{\beta_n \dotp{\uu}{(n\kk-\mm)}} \PProb_n[\XX=n\kk]\\
	  &    	\leq |A_n|e^{\beta_n \sigma_n}\PProb_n[\XX = n\kk].
\end{align*}
    Since $\sigma_n = n^{(d+2)/(2d+2)}$, the factor $|A_n| = O(\sigma_n^d)$ grows only as a power of $n$ while $e^{\beta_n\sigma_n} = e^{O(n^{d/(2d+2)})}$. Since $\PProb_n[\XX \in A_n] \geq c$, the result is proven.
\end{proof}

\subsection{The number of generators of $T \in \T(C,n\kk)$}
\label{ssec:number_generators}

Zonotopes in $\T(C,n\kk)$ have been identified with functions $\omega: \PPrim \cap C \to \ZZ_+$ with finite support. The {\sl generators} of $T \in \T(C,n\kk)$ are those $\xx \in \PPrim$ for which $\omega(\xx)>0$, we write $G(T)$ for the set of generators of $T$. Using the method of the previous subsection we determine the expected number $\mu_n$ of the number of generators of  $T \in \T(C,n\kk)$, under the distribution $\mathrm \PProb_n$. By Proposition~\ref{prop:zono_local}, it will also lead to an estimate under the uniform distribution $\mathrm Q_n$ on $\T(C_n\kk)$.

\begin{lemma}\label{l:generator}
\[
\mu_n=\EE[|G(T)|]=\frac 1{\zeta(d)\beta_n^d}\Lambda_C(\uu)(1+O(\beta_n)).
\]
\end{lemma}

\begin{proof}Let $\beta=\beta_n$ in the proof. Formula (\ref{eq:geomdistr}) implies that
$$\PProb_n[\omega(\mathbf{x}) > 0] = e^{-\beta \dotp{ \mathbf{u}}{ \mathbf{x} }}.$$

A consequence of Proposition~\ref{prop:primitive_sums} is that
$$\left|\sum_{x \in \mathbb{Z}^d\cap C} e^{-\beta \dotp{ \mathbf{u}}{ \mathbf{x} }} - \frac{1}{\zeta(d)} \int_C e^{-\beta \dotp{ \mathbf{u}}{ \mathbf{x} }} d\mathbf{x} \right| \leq \frac{c_d}{\beta^{d - 1}},$$
for some constant $c_d> 0$ depending on the dimension. Here
$$
\int_C e^{-\beta \dotp{ \mathbf{u}}{ \mathbf{x} }} d\mathbf{x}=\frac 1{\beta^d}\int_C e^{ \dotp{ \mathbf{u}}{ \mathbf{x} }} d\mathbf{x}=\frac 1{\beta^d}\Lambda_C(u),
$$
and so
$$\mu = \frac{1}{\zeta(d)\beta^d}\Lambda_C(\mathbf{u})(1+ O(\beta)).$$
\end{proof}

As $\beta_n^{-d}=\Theta\left( n^{d/(d+1)}\right)$ follows from the definition of $\beta$, the order of magnitude of $\mu$ is $n^{d/(d+1)}$.

We will need a similar result for the number of generators lying in a hyperplane $H$.

\begin{lemma}\label{l:generatorinH} Assume $H$ is a hyperplane in $\RR^d$ containing the origin. Then
\[
\mu_H=\EE[|G(T)\cap H|]=O(n^{(d-1)/(d+1)}),
\]
where the implicit constant in $O(\cdot)$ is independent of $H$.
\end{lemma}

\begin{proof}
  Let $\beta=\beta_n$ in the proof.
Following the above argument, we have
$$
\mu_H = \EE [|G(T) \cap H|] = \sum_{\mathbf{x} \in \mathbb{Z}^d\cap C \cap H} e^{-\beta \dotp{ \mathbf{u}}{ \mathbf{x} }}.
$$
The one-sided estimate of Lemma~\ref{lem:one_sided} gives now that
$$
\sum_{\mathbf{x} \in \mathbb{Z}^d\cap C \cap H} e^{-\beta \dotp{ \mathbf{u}}{ \mathbf{x} }}
\leq
\zeta(d)^{-1}\int_{C \cap H} e^{-\beta \dotp{ \mathbf{u}}{ \mathbf{x} }}d\xx + \frac{c_{d-1}}{\beta^{d-2}}.
$$
Here
$$
\int_{C \cap H} e^{-\beta \dotp{ \mathbf{u}}{ \mathbf{x} }} d\xx
= \frac 1{\beta^{d-1}}\int_{C \cap H} e^{ \dotp{ \mathbf{u}}{ \mathbf{x} }} d\xx.
$$
Thus indeed
$$\mu_{H} \leq \frac{1+o(1)}{\zeta(d)\beta^{d-1}}\int_{C \cap H}e^{-\dotp{ \mathbf{u}}{ \mathbf{x} }} d\xx =O(n^{(d-1)/(d+1)}).
$$
\end{proof}

\subsection{Limit shapes}

In this section, we are going to show the existence of a limit shape for random zonotopes drawn under the uniform distribution on the set of all lattice zonotopes in the cone $C$ with endpoint $n\kk$.

For every convex compact subset $A$ of $\RR^d$, the \emph{support function} of $A$ is the continuous function $h_A \colon \RR^d \to \RR$ defined by
\[
    h_A(\ww) = \sup \{\dotp{ \vv}{\ww } ,\; \vv \in A\}, \qquad \ww \in \RR^d.
\]
Let $T_0=T_0(C,\kk)$ be the zonoid defined in Section~\ref{sec:limiting_zonoid}.

\begin{theorem}
    \label{thm:zono_shape}
    For all $\vv \in \RR^d$ and for all $\epsilon > 0$,
    \[
        \lim_{n\to+\infty}\mathrm Q_n\left[|h_{\frac{1}{n}T}(\vv) - h_{T_0}(\vv)| > \epsilon \right] = 0.
    \]
\end{theorem}

\begin{proof}
Again, the proof of this theorem is based on the probabilistic model of Section~\ref{ssec:model}. It depends on the fact that, conditional on the event $\XX = n\kk$, the distribution of a random zonotope under $\PProb_n$ is uniformly distributed on the set $\T(C,n\kk)$. In particular,
\begin{align*}
	\mathrm Q_n\left[ |h_{\frac{1}{n}T}(\vv) - h_{T_0}(\vv)| > \epsilon \right] & = \PProb_n\left[|h_{\frac{1}{n}T}(\vv) - h_{T_0}(\vv)| > \epsilon \mid \XX = n\kk\right]\\
        & \leq \frac{\PProb_n\left[|h_{\frac{1}{n}T}(\vv) - h_{T_0}(\vv)| > \epsilon\right]}{\PProb_n\left[\XX = n\kk\right]}.
\end{align*}
Hence, we only need to prove that the right-hand side of this inequality goes to $0$. It is natural to consider only $\vv \in \BB^{d-1}$, the unit sphere of $\RR^d$. Using Proposition~\ref{prop:zono_local}, we see that the theorem follows from the result of the next proposition.
\end{proof}

\begin{proposition}\label{thm:zono_shape2}
    For all $\epsilon > 0$ there exists $c > 0$ such that for all $\vv \in \BB^{d-1}$ and for all $n$ large enough,
    \[
        \PProb_n\left[ |h_{\frac{1}{n}T}(\vv) - h_{T_0}(\vv)| > \epsilon \right] \leq \exp\{-c\, n^{d/(d+1)}\}.
    \]
\end{proposition}

\begin{proof}
 It is evident that for a zonotope $T(\omega)$ generated by the vectors $\xx \in \PPrim\cap C$ with multiplicities $\omega(\xx)$ the support function is
    \[
        h_{T}(\vv) = \sum_{\xx \in \PPrim \cap C} \omega(\xx)\,\dotp{ \xx}{\vv}\; \indicator{\dotp{\xx}{\vv} \geq 0} = \sum_{\xx \in \PPrim \cap C(\vv \geq 0)} \omega(\xx)\,\dotp{\xx}{\vv}.
    \]
    In the proof of Proposition~\ref{prop:central_limit}, we could replace $C$ with the smaller cone $C(\vv \geq 0)$ and we would obtain
    \[
        \frac{1}{n}\,\EE_n[h_{T}(\vv)] = \dotp{\Lambda_{C(\vv \geq 0)}(\uu)}{\vv} + O(n^{-1/(d+1)}) = h_{T_0}(\vv) + o(1).
    \]
    Since $\frac{1}{n}h_{T} = h_{\frac{1}{n}T}$, this already shows that the limit shape appears in expectation.

    We will now bound the probability for a large deviation from the mean using the so-called Chernoff method. Let us consider the exponential generating function of $h_T(\vv)$
    \[
        \EE[e^{\theta h_{T}(\vv)}] =
        \prod_{x \in \PPrim \cap C(\vv \geq 0)} \frac{1-e^{-\dotp{\beta\uu}{\xx}}}{1-e^{-\dotp{(\beta\uu - \theta \vv)}{ \xx}}}
            = \frac{Z_{C(\vv \geq 0)}(\beta \uu - \theta \vv)}{Z_{C(\vv \geq 0)}(\beta\uu)}
    \]
    which is well defined for all $\theta$ small enough with respect to $\beta$.
    Consider now the centered random variable $Y = h_{\frac{1}{n}T}(\vv) - \EE[h_{\frac{1}{n}T}(\vv)]$.
    The central limit part of Proposition~\ref{prop:central_limit} applies to the cone $C(\vv\ge 0)$ with $\kk'=-\nabla \Lambda_{C(\vv\ge 0)}(\uu)$. Thus by second order Taylor approximation, we obtain for some constant $c(\uu,\vv) > 0$ involving the Hessian matrix of $\Lambda_{C(\vv \geq 0)}$ at $\uu$,
    \[
        \log \EE [e^{\theta Y}] \sim c(\uu,\vv)\,\frac{\theta^2}{2} n^{-d/(d+1)}
    \]
    as long as $\frac{\theta}{n}$ goes to $0$. But for all $\theta > 0$, the Markov inequality yields
    \[
        \PProb_n[Y > \epsilon] = \PProb_n[\theta Y \geq \theta\epsilon] \leq e^{-\theta\epsilon} \EE[e^{\theta Y}]
    \]
    This bound is approximately optimized for $\theta = c(\uu,\vv)^{-1}\, \epsilon\, n^{d/(d+1)}$ and it leads to
    \[
        \PProb_n[Y > \epsilon] \leq \exp\left(-\frac{1}{2}c(\uu,\vv)^{-1} \epsilon^2 n^{d/(d+1)}(1+o(1))\right)
    \]
    A similar bound holds for $\PProb_n[-Y > \epsilon]$, hence for
    $\PProb_n[|Y| > \epsilon]$.
    The conclusion follows.
\end{proof}

\subsection{The number of zonotopes}
\label{ssec:number_zonotopes}

This section is devoted to the proof of Theorem~\ref{th:size}, which states that
\[
    \frac{\log p(C,n\kk)}{n^{\frac{d}{d+1}}} \xrightarrow[n\to\infty]{} c_d\,q(C,\kk),\qquad \text{where } c_d = \sqrt[d+1]{\frac{\zeta(d+1)}{\zeta(d)}(d+1)!} 
\]
and where $q(C,\kk)$ is defined in Section~\ref{sec:limiting_zonoid} as the minimal volume of a cap of the cone $C$ having $\kk$ as its center of gravity.

The starting point of the proof is the equation~\eqref{eq:proba} established earlier which links the number $\log p(C,\kk)$ with the probability distribution $\PProb_n$ in the following way:
\[
  \log p(C,\kk) = \log Z_n + n\beta_n \uu\cdot k + \log \PProb_n[\XX = n\kk].
\]
The value of $\beta_n$ was given explicitly by \eqref{eq:param} and it shows that $n\beta_n$ is of order $n^{d/(d+1)}$. On the other hand, we know by Proposition~\ref{prop:zono_local} that $\log \PProb_n[\XX = n\kk]$ is negligible compared to $n^{d/(d+1)}$.

It only remains to estimate the term $\log Z_n$. This is done by using the series expansion \eqref{eqn:logZ} and applying Proposition~\ref{prop:primitive_sums} from the appendix in the same lines as the first part of the proof of Proposition~\ref{prop:central_limit}. This leads to
\[
  \log Z_n = \frac{\zeta(d+1)}{\zeta(d)} \frac{1}{\beta_n^d}(\Lambda(\uu)+o(1)),
\]
which is again of order $n^{d/(d+1)}$. A simple computation based on the definition of $\beta_n$ and of the constant $c_d$ yields therefore
\[
  \frac{\log p(C,n\kk)}{n^{\frac{d}{d+1}}} \xrightarrow[n\to\infty]{} \frac{c_d}{\sqrt[d+1]{(d+1)!}} (\Lambda(\uu)+\uu\cdot \kk).
\]

The final step is to express $\Lambda(\uu) + \uu\cdot \kk$ in terms of $q(C,\kk)$. Recall that $\uu$ is defined by \eqref{eq:param} and that we have shown in Proposition~\ref{prop:central_limit} that
\[
  \kk  = (d+1)! \int_{C(\uu \leq 1)} \xx\,d\xx= \int_C \xx e^{-\uu\cdot\xx}\,d\xx.
\]
The computation of $\Lambda(\uu)+\uu\cdot\kk$ is thus simply a matter of integral calculus. Using the fundamental theorem of calculus and the Fubini-Tonelli theorem, we obtain:
\begin{align*}
\Lambda(\uu) + \uu\cdot \kk &= \int_C (1+\uu \cdot \xx)e^{-\uu\cdot x}\,d\xx\\
& = \int_C \int_{\uu\cdot \xx}^\infty t\,e^{-t} \,dt\,d\xx\\
& = \int_0^\infty  t e^{-t}\Vol C(\uu \leq t)\,dt\\
& = (d+1)!\Vol C(\uu \leq 1) 
\end{align*}
Finally, the definition \eqref{eq:param} of $\uu$ and the considerations of volumes in Section~\ref{sec:limiting_zonoid} show that
\[
  \Vol C(\uu \leq 1) = q(C,\tfrac{1}{(d+1)!}\kk) = \left[\frac{1}{(d+1)!}\right]^{\frac{d}{d+1}} q(C,\kk).
\]
The proof is complete.

\medbreak
\noindent
\textit{Remark.} Although we did not use it directly, it appears above that $\uu$ is actually the unique minimizer of the quantity $\Lambda(\vv) + \vv\cdot\kk$ for $\vv \in C^\circ$. This fact can be used to prove Gigena's Theorem~\ref{th:gigena}.

\section{Number of vertices of a random zonotope}
\label{sec:number_vertices}

In this section we determine the order of magnitude of the expected number of faces of each dimension and towers of a random zonotope.
A tower (or flag) of a zonotope $T$ is a chain $F_0 \subset F_1 \subset \ldots \subset F_{d-1}$, where $F_i$ is an $i$-dimensional face of $T$.
We denote by $f_i(T)$ the number of $i$-dimensional faces and by $F(T)$ the number of towers of $T$.

\begin{theorem}\label{thm:typicalZonotopeVertices}
There are constants $c_1, c_2$ not depending on $n$ such that
\begin{align*}
\lim_{n \rightarrow + \infty} Q_n\left\{T\in \T(C,n\mathbf{k}) \mid c_1 < \frac{f_i(T)}{ n^{d(d-1)/(d+1)}} < c_2\right\} &= 1, \\
\lim_{n \rightarrow + \infty} Q_n\left\{\T \in \T(C,n\mathbf{k}) \mid c_1 < \frac{F(T)}{ n^{d(d-1)/(d+1)}} < c_2 \right\} &= 1.
\end{align*}
\end{theorem}

As $Q_n$ is the uniform distribution on $\T(C,n\kk)$ the above theorem says that, as $n\to \infty$, for the overwhelming majority of the zonotopes in $\T(C,n\kk)$  the face numbers $f_i(T)$ and $F(T)$ are of order $n^{d(d-1)/(d+1)}$. It is shown in \cites{barany_larman_1998, konyagin1984estimation} that for any integer lattice polytope $P$ with $\vol P>0$, $$F(P) = O\left((\vol P)^{(d-1)/(d+1)}\right).$$ The same upper bound for each $f_i(P)$ follows immediately. Combined with the previous result on the volume of a random zonotope, this establishes the upper bounds of Theorem \ref{thm:typicalZonotopeVertices}. It only remains to show the lower bounds.

\subsection{Generators of a random zonotope}

Let $\T_\varepsilon(C,n\mathbf{k})$ consist of all zonotopes $T\in \T(C,n\mathbf{k})$ such that
$$|G(T)| \geq (1-\varepsilon)\zeta(d)^{-1}\Lambda_C(\beta \mathbf{u}),$$
and the number of generators of $T$ that are contained in any single hyperplane is at most $\varepsilon \Lambda_C(\beta \mathbf{u})$.
Our goal in this section is to prove all except for a minute fraction of the zonotopes in $\T(C,n\mathbf{k})$ belong to $\T_\varepsilon(C,n\mathbf{k})$.

\begin{proposition}\label{thm:typicalZonotopeGenerators}
$$\lim_{n \rightarrow + \infty} Q_n[\T_\varepsilon(C,n\mathbf{k})] = 1.$$
\end{proposition}

\begin{proof} This is quite easy. We rely on the two lemmas in Subsection~\ref{ssec:number_generators} and on the following rough form of the Chernoff bounds:
\begin{proposition}[Chernoff bounds]
Let $X_1, \ldots, X_n$ be independent Bernoulli random variables, let $X = X_1 + \ldots + X_n$, and let $\mu = \EE[X]$.
For any $\delta > 0$,
\begin{align*}
\PProb[X \geq (1+\delta)\mu] & \leq e^{-\delta^2 \mu / 3}, & 0 < \delta < 1, \\
\PProb[X \geq (1+\delta)\mu] & \leq e^{-\delta \mu / 3}, & 1 \leq \delta, \\
\PProb[X \leq (1-\delta)\mu] & \leq e^{-\delta^2 \mu / 2},&  0 < \delta < 1.
\end{align*}
\end{proposition}

By Lemma~\ref{l:generator} $\mu= \EE[|G(T)|]$ is of order $n^{d/(d+1)}$, and so
an application of the Chernoff bound gives, that for any $0 < \varepsilon < 1$
$$\PProb_n\left[|G(T)| \leq (1-\varepsilon) \mu\right] \leq  e^{-\varepsilon^2\mu/2} =
\exp\{-cn^{d/(d+1)}\}$$
where $c>0$ is a constant independent of $n$.

Next, we bound the probability that a hyperplane $H$ with $0\in H$ contains more than a constant fraction of the expected number of generators of a zonotope in $\T(C,n\mathbf{k})$. According to Lemma~\ref{l:generatorinH}, $\mu_H =\EE[|G(T)\cap H|]\le c n^{(d-1)/(d+1)}$ with $c$ not depending on $H$ and $n$. Applying the Chernoff bound with $\gamma > 1$, we have
$$\PProb_n \left[|G(T) \cap H| \geq (1+\gamma)\mu_{H} \right] \leq e^{-\gamma \mu_H/3}.$$
In particular,
$$\PProb_n \left[|G(T) \cap H| \geq \varepsilon \Lambda_C(\beta \mathbf{u})\right] \leq \exp\{-cn^{d/(d+1)}\}$$
for some $c>0$ depending on $\varepsilon$.

Note finally that the total number of hyperplanes (containing the origin) spanned by the generators of a zonotope with $O(n^{d/(d+1)})$ generators is bounded by $O(n^{d(d-1)/(d+1)})$.
Since this is polynomial in $n$, a union bound proves Proposition \ref{thm:typicalZonotopeGenerators}.
\end{proof}

\subsection{Faces of a random zonotope}

In this section, we show that the zonotopes in $\mathcal{T}_\varepsilon(C, n\mathbf{k})$ have at least the number of $i$-dimensional faces  prescribed by Theorem \ref{thm:typicalZonotopeVertices}; the same lower bound for towers is an immediate corollary.

We rely on the following well-known duality between zonotopes and hyperplane arrangements.
An arrangement of hyperplanes is a finite collection of affine hyperplanes, together with their subdivision of $\mathbb{R}^d$ into relatively open cells. A central hyperplane arrangement is the decomposition of $\mathbb{R}^d$ induced by linear hyperplanes.
For a $d$-dimensional zonotope $T$ with generators $\{\mathbf{v}_1, \ldots, \mathbf{v}_m\}$, let $\mathcal{A}^c(T)$ be the central arrangement of hyperplanes with this set of normal vectors.
The number of $k$-dimensional cells of $\mathcal{A}^c(T)$ is equal to the number of $d-k$ faces of $T$ (see, e.g., \cite{bjorner1999oriented}*{Prop. 2.2.2}).
Let $\mathcal{A}(T)$ be the $(d-1)$-dimensional arrangement of affine hyperplanes obtained as the intersection of $\mathcal{A}^c(T)$ with a generic hyperplane. Clearly, the number of $(k-1)$-dimensional cells of $\mathcal{A}(T)$ is a lower bound on the number of $k$-dimensional cells of $\mathcal{A}(T)$ for $k \in [1,d]$. We note that twice this lower bound is an upper bound on the number in question as the cells of $\mathcal{A}^c(T)$ come in pairs $C,-C$ except $C=\{0\}$.

\begin{proposition}\label{thm:vertices} There is a constant $c$ independent of $n$ such that for all $\epsilon \in (0, c)$ and $T \in \T_{\epsilon} (C, n\mathbf{k})$, we have $f_i(T) = \Omega(n^{d(d-1)/(d+1)})$ for each $0 \leq i \leq d-1$.
\end{proposition}

\begin{proof}
If a single line $\{\lambda \vv \mid \lambda \in \RR\}$ is incident to $k$ hyperplanes of $\mathcal{A}^c(T)$, then the vector $\vv$ is orthogonal to $k$ generators of $T$. Since at most $(\varepsilon/(1-\varepsilon))|G(T)|$ of the generators of $T$ lie in the hyperplane $\vv^{\perp}$ (for any $\vv$), the same bound applies to the maximum number of hyperplanes of $\mathcal{A}(T)$ that are incident to any single point.
Hence, we can apply the following theorem of Beck \cite{beck} to bound the number of vertices of $\mathcal{A}(T)$.

\begin{lemma}
There is a constant $c$ depending on dimension such that the following holds.
Let $\mathcal{A}$ be an arrangement of $m$ hyperplanes in $\mathbb{R}^{d-1}$.
Then either
\begin{enumerate}
\item a single vertex of $\mathcal{A}$ is incident to $cm$ hyperplanes of $\mathcal{A}$, or
\item the total number of vertices in $\mathcal{A}$ is $\Omega(m^{d-1})$.
\end{enumerate}
\end{lemma}

Since $\mathcal{A}(T)$ is an arrangement of $n^{d/(d+1)}$ hyperplanes, this implies that, for $\varepsilon$ sufficiently small, the number of vertices of $\mathcal{A}(T)$ is $\Omega(n^{d(d-1)/(d+1)})$.

Designate a generic direction $\mathbf{w}$ in $\mathbb{R}^{d-1}$ to be ``up''.
Clearly, each $(d-1)$-dimensional region of $\mathcal{A}(T)$ that is bounded from below has a bottom vertex.
We claim that each vertex of $\mathcal{A}(T)$ is the bottom vertex of at least one region.
Let $p$ be an arbitrary vertex of $\mathcal{A}(T)$.
At least $d-1$ hyperplanes having linearly independent normals are incident to $p$. Let $S$ be an arbitrary set of $d-1$ such hyperplanes.
At least one $(d-1)$ dimensional region $R$ of $\mathcal{A}(T)$ is above each of these hyperplanes with respect to $\mathbf{w}$, and $p$ is the bottom vertex of each such region.
In addition, $p$ is the bottom vertex of each $j$-dimensional region that is contained in a hyperplane of $S$ and that bounds $R$, for $1 \leq j \leq d-2$.
Hence, the number of $j$-dimensional regions that are bounded below in $\mathcal{A}(T)$ is at least the number of vertices of $\mathcal{A}(T)$, which is $\Omega(n^{d(d-1)/(d+1)})$.

Since the $j$-dimensional regions of $\mathcal{A}(T)$ are in bijection with the $(d-j)$-dimensional faces of $T$, this completes the proof of the proposition.
\end{proof}

\subsection{A short digression}

This proof method has an interesting consequence about hyperplane (or rather subspace) arrangements that seems to be new. Recall first that if $\mathcal{A}$ is a central arrangement of $m$ general position hyperplanes in $\mathbb{R}^d$, then $f_i(\mathcal{A})$ is known precisely for all $i$:
\[
f_i(\mathcal{A})=\sum_{k=d-i}^k {k \choose d-i}{m \choose k}
\]
see for instance  Buck~\cite{buck_1943} and Zaslavski~\cite{zaslavsky_facing_1975}.  Moreover,  this function is an upper bound on the number of $i$-dimensional cells of every central arrangement of $m$ linear hyperplanes in $\mathbb{R}^d$. Define $\mathcal{A}^c_r$ as the collection of all $d-1$-dimensional subspaces of the form $\mathbf{z}^{\perp}$ where $\mathbf{z} \in \ZZ^d$ has Euclidean length at most $r$.
We could not find the following result anywhere in the literature or in folklore.

\begin{theorem} The number of $i$-dimensional cells  for all $i=1,\ldots,d$ and the number of towers of $\mathcal{A}_r$ is $\Theta(r^{d(d-1)})$.
\end{theorem}

\begin{proof} It follows from the result of Beck~\cite{beck} cited above that the number of one-dimensional cells of $\mathcal{A}^c_r$ is $\Theta(r^{d(d-1)})$. The same estimate is also implied by Theorem 3 of \cite{barany_harcos_2001}.

Let $\mathcal{A}_r$ be the $(d-1)$-dimensional arrangement of affine hyperplanes obtained as the intersection of $\mathcal{A}^c_r$ with a general position hyperplane $H$ the same way as above. The previous proof applies word by word.
\end{proof}

\section{Integral zonotopes in convex bodies}\label{sec:convex_bodies}

As it is mentioned in the introduction, for $d=2$ Theorem~\ref{th:lshape} follows from the results of \Barany{}~\cite{barany'97}. In fact, more is true. We set up the question more generally. Assume $K \subset \RR^d$ a convex body (i.e., convex compact set with non-empty interior) and write $\PP(K,n)$ for the collection of all convex $\frac 1n \ZZ^d$-lattice polytopes contained in $K$, $\PP(K,n)$ is a finite set. It is proved in \cite{barany'97}  that, when $d=2$, the polygons in $\PP(K,n)$ have a limit shape as $n\to \infty$:

\begin{theorem}\label{th:lshap} For every convex body $K$ in $\RR^2$ there is a convex body $K_0\subset K$ such that, as $n\to \infty$, the overwhelming majority of the polygons in $\PP(K,n)$ are very close to $K_0$. More precisely for every $\epsilon>0$
\[
\lim_{n \to \infty}\frac{\operatorname{card}\left\{P \in \PP(K,n) \mid \dist (P,K_0)>\epsilon\right\}}{\operatorname{card}\PP(K,n)}=0.
\]
\end{theorem}

The distinguishing property of $K_0$ is that its affine perimeter (for the definition see \cite{barany'97}) is maximal among all convex subsets of $K$. It is also shown there that such a $K_0$ is unique. The analogous question in higher dimension is wide open. Even the case when $K$ is the cube $[-1,1]^3$ is not known.

\medskip 
{\bf Question 1.} Assume $K \subset \RR^d$ is a convex body and $d>2$. Is there a limit shape to the convex polytopes in $\PP(K,n)$?

\medskip
Define, similarly, $\F(K,n)$ as the collection of all convex $\frac 1n \ZZ^d$-lattice zonotopes contained in $K$, $\F(K,n)$ is again a finite set.

\medskip
{\bf Question 2.} Assume $K \subset \RR^d$ is a convex body and $d\ge 2$. Is there a limit shape to the convex zonotopes in $\F(K,n)$?

\medskip
We can answer this question when $d=2$. Informally stating, for every convex body $K$ in the plane there is a zonoid $K^0$ contained in $K$ such that, as $n\to \infty$, the overwhelming majority of the polygons in $\F(K,n)$ are very close to $K^0$. The {\bf proof} follows the method of \cite{barany'97} so we only give a sketch. Note that in the plane a zonoid is always a centrally symmetric convex body and vice versa. The first thing to show is that $K$ contains a centrally symmetric convex body that maximizes the affine perimeter among all centrally symmetric convex subsets of $K$. This follows from the fact that the affine perimeter is upper semi-continuous. The next step is to prove that the maximizer is unique. Here one uses the fact that, for all $t\in [0,1]$, if $A,B \subset K$ are centrally symmetric convex bodies, then so is $tA+(1-t)B$ which is a subset of $K$, again. The affine perimeter is a concave functional, that is, writing $AP(S)$ for the affine perimeter of the convex body $S$ we have
\[
AP(tA+(1-t)B) \ge tAP(A)+(1-t)AP(B). 
\]
This shows that if $A$ and $B$ are both maximizers, then so is $tA+(1-t)B$ for all $t \in [0,1]$. Next one proves that $A \subset B$ cannot happen. Then $A$ and $B$ are both maximizers not only for $K$ but also for the convex body $\Conv (A \cup B)$. Then $A=B$ follows more or less the same way as in \cite{barany'97}. The proof that the unique maximizer is the limit shape in $\PP(K,n)$ goes analogously to  \cite{barany'97}. 

\medskip
We can also answer Question 2 when the convex body $K$ is the cube $Q=Q^d=[-1,1]^d$ for all $d\ge 2$. To describe the limiting zonoid in this case, set $t=\sqrt[d+1]{(d+1)!2^{1-d}}$ and define the octahedron 
\[
O^d=t\Conv \{\pm e_1,\ldots,\pm e_d\}. 
\]
With this notation the support function of the limiting zonoid $Q_0$ is given by
\[
h_{Q_0}(\vv)=\frac 12 \int_{O^d}|\xx\cdot \vv|d\xx.
\]

\begin{theorem} With the above notation for every $d\ge 2$ and for every $\eta>0$
\[
\lim_{n \to \infty}\frac{\operatorname{card}\left\{P \in \F(Q^d,n) \mid \dist (P,Q_0)>\eta\right\}}{\operatorname{card}\F(Q^d,n)}=0.
\]
\end{theorem}

\begin{proof} The coordinate hyperplanes $x_i=0$ for $i=1,\ldots,d$ split $\RR^d$ into $2^d$ open cones; each corresponds to a sign pattern $\eps=(\eps_1,\ldots,\eps_d)\in \{-1,1\}^d$. Extend the cone with sign pattern $\eps$ to another cone, to be denoted by $C^{\eps}$ so that these $2^d$ cones form a partition of $\RR^d$. This is clearly possible.

As we have seen in Section~\ref{sec:strict_partitions} for every zonotope $T \in \PP(Q,n)$ there is a unique multiset $W\subset \frac 1n \PPrim$ such that $T=T(W)$. Define $W^{\eps}=W\cap C^{\eps}$ for every $\eps$ and set $\kk^{\eps}=\sum_{w \in W^{\eps}}w$. Then $T^{\eps}=T(W^{\eps})$ is a zonotope in the cone $C^{\eps}$ with endpoint $\kk^{\eps}$ and containing the origin.  Moreover, $T=\sum_{\eps}T^{\eps}$. As $T\subset Q$, the conditions
\begin{equation}\label{eq:cond}
\sum_{\eps: \eps_i=1}k_i^{\eps}\le 1 \mbox{ and }\sum_{\eps: \eps_i=-1}k_i^{\eps}\ge -1
\end{equation}
are satisfied for every $i=1,\ldots,d$.

A fixed $\kk^{\eps}\in C^{\eps}\cap \frac 1n \ZZ$ defines a family $\F(\kk^{\eps})$ consisting of all zonotopes in $C^{\eps}$ that contain the origin and whose endpoint is $\kk^{\eps}$.
Given $\kk^{\eps} \in C^{\eps}$ for all $\eps$ we write $\kk=(\kk^{\eps}: \mbox{ all } \eps)$ which is in fact a $d\times 2^d$ matrix. Let $\F(\kk)$ be the collection of all zonotopes $T \in \F(Q,n)$ for which the endpoint of $T^{\eps}$ is $\kk^{\eps}$ for all $\eps$. 

After these preparation here is the plan of the proof. We show that the contribution to $|\F(Q,n)|$ of $\F(\kk)$ is minute unless $|k_i^{\eps}- \eps_i 2^{1-d}|$ is very small for all $\eps$ and all $i$. This will prove that for most zonotopes $T \in \F(Q,n)$ and for all $\eps$ the endpoint of $T^{\eps}$ is very close to the vector $(\eps_12^{1-d},\ldots,\eps_d2^{1-d})$. The limit shape of these zonotopes is very close to the limiting zonoid of the zonotopes in $C^{\eps}$ with endpoint $(\eps_12^{1-d},\ldots,\eps_d2^{1-d})$ which is known from Example 1. Finally, the sum of the limiting zonoids is $Q_0$ as one can check easily.

It follows from Theorem~\ref{th:size} and Example 1 that
\[
\log |\F(\kk^{\eps})|=c_d\left(\frac {(d+1)^d}{d!} \prod_1^d|k_i^{\eps}|\right)^{1/(d+1)}n^{d/(d+1)}(1+o(1)).
\]
Since $\log |\F(\kk)|=\sum_{\eps} \log |\F(\kk^{\eps})|$ we will have to deal with the expression
\[
S(\kk):= \sum_{\eps}\left(\prod_1^d|k_i^{\eps}|\right)^{1/(d+1)}. 
\]
Writing $c_d'=c_d\sqrt[d+1]{\frac {(d+1)^d}{d!}}$, we have
\[
\log \F(\kk)=c'_d S(\kk)n^{d/(d+1)}(1+o(1)).
\]

Our first task is to show that 
\begin{equation}\label{eq:upp}
S(\kk) \le \frac d{d+1}2^{2d/(d+1)}(1+2^{1-d})=:D_d.
\end{equation}

To see this it is convenient to define $k_0^{\eps}=2^{1-d}$. Then, using the inequality between the arithmetic and geometric means $2^d$ times and the conditions (\ref{eq:cond})
\begin{eqnarray*}
S(\kk)&=&2^{(d-1)/(d+1)}\sum_{\eps}\left(\prod_0^d|k_i^{\eps}|\right)^{1/(d+1)} \le 2^{(d-1)/(d+1)}\frac1{d+1}\sum_{\eps}\left(\sum_0^d|k_i^{\eps}|\right)\\
      &=&2^{(d-1)/(d+1)}\frac 1{d+1}\sum_0^d(\sum_{\eps}|k_i^{\eps}|) \le 2^{(d-1)/(d+1)}\frac 2{d+1}\sum_0^d(1+2^{1-d})\\
      &=& \frac d{d+1}2^{2d/(d+1)}(1+2^{1-d})=D_d.
\end{eqnarray*}      
Note that equality holds throughout if all $k_i^{\eps}=\pm 2^{1-d}$, so $\max S(\kk)=D_d$. 

Here $\kk^{\eps}$ can be chosen in at most $(n+1)^d$ ways, so $\kk$ can take at most $(n+1)^{d2^d}$ values. Then
\[
|\F(Q,n)| = \sum_{\kk}\exp\{c_d'S(\kk)n^{d/(d+1)}(1+o(1))\} \le \exp\{c_d'D_dn^{d/(d+1)}(1+o(1))\}. 
\]

We are to use a strengthening of the inequality between the arithmetic and geometric means. Let $t_1,\ldots,t_m$ be non-negative reals (not all zero) and let $A_m$ resp. $G_m$ denote their arithmetic and geometric mean. Then, for any $i,j$ in $\{1,\ldots,m\}$
\begin{equation}\label{eq:AG}
\frac {(t_i-t_j)^2}{2m^2A_m}\le A_m-G_m.
\end{equation}
For the reader's convenience we prove this inequality at the end of this section.

Assume next that $|k_j^{\eps}-\eps_j2^{1-d}|\ge \delta$ for some $\eps$ and for some $j=1,\ldots,d$ where $\delta>0$. Then using (\ref{eq:AG}) in the above application of the inequality between the arithmetic and geometric means in the term corresponding to this $\eps$ we have can add an extra term, namely
\begin{eqnarray*}
(\prod_0^d|k_i^{\eps}|)^{1/(d+1)} &\le& \frac1{d+1}\sum_{\eps}(\sum_0^d|k_i^{\eps}|)-\frac {(k_j^{\eps}-\eps_j2^{1-d})^2}{2(d+1)^2(1+2^{1-d})}\\
&\le&  \frac1{d+1}\sum_{\eps}(\sum_0^d|k_i^{\eps}|)- \frac {\delta^2}{3(d+1)^2}.
\end{eqnarray*}
Here we used the fact that $\sum_{j=0}^d k_j^{\eps}\le 1+2^{1-d}$. This implies that when $|k_i^{\eps}-\eps_i2^{1-d}|\ge \delta$ for some positive $\delta$, then 
\[
S(\kk) \le D_d-\frac {\delta^2}{3(d+1)^2}.
\]
This shows in turn that the contribution of the $\F(\kk)$ with all such $\kk$ to $|\F(Q,n)|$ is only 
\[
\exp\{c_d'(D_d-\delta^2/3(d+1)^2)n^{d/(d+1)}(1+o(1))\},
\]
much smaller than $|\F(Q,n)|$ as long as $\delta^2n^{d/(d+1)}$ is larger than any fixed positive constant, say $\eta^2$. Thus, for all $\delta > 0$, as $n\to\infty$, all but a small fraction of the zonotopes in $\F(Q,n)$ satisfy $|k_i^{\eps}-\eps_i2^{1-d}|\le \delta$ for all $\eps$ and all $i$. Theorem~\ref{th:lshape} and Example 1 show that, again, all but a small fraction of the zonotopes in $\F(\kk^{\eps})$ are very close to the limiting zonoid in $C^{\eps}$ corresponding to $\kk^{\eps}=(\eps_12^{1-d},\ldots,\eps_d2^{1-d})$.
\end{proof}

\medskip
\begin{proof}[Proof of \eqref{eq:AG}] We may assume that $i=1$ and $j=2$. Suppose $m\ge 3$ and let $A_{m-1}$ resp. $G_{m-1}$ denote the arithmetic and geometric mean of $t_1,\ldots,t_{m-1}$. Richard Rado~\cite{hardy_inequalities_1952}*{Theorem 60} proved the inequality
\[
(m-1)(A_{m-1}-G_{m-1})\le m(A_m-G_m),
\]
which is, of course, a strengthening of the inequality between the arithmetic and geometric means. Then for $m\ge 2$ 
\begin{eqnarray*}
m(A_m-G_m) &\ge& 2(A_2-G_2)= (\sqrt {t_1}-\sqrt {t_2})^2= \frac {(t_1-t_2)^2}{(\sqrt {t_1}+\sqrt {t_2})^2}\\
&\ge& \frac {(t_1-t_2)^2}{2(t_1+t_2)} \ge \frac {(t_1-t_2)^2}{2mA_m}.
\end{eqnarray*}
\end{proof}
 
\appendix
\section{Discrete sums and integrals on cones}

The aim of this section is to establish the following proposition.

\begin{proposition}
\label{prop:primitive_sums}
  Let $d \geq 3$ and let $\PPrim = \{\xx \in \ZZ^d \mid \gcd(x_1,\dots,x_d) = 1\}$ denote the set of primitive vectors.
Let $f \colon C \to \RR$ be continuously differentiable and positively homogeneous of degree $h$, that is to say $f(\lambda \xx) = \lambda^h f(\xx)$ for all $\xx \in C$ and $\lambda \geq 0$.
For every $\uu \in \Int(C^\circ)$,
    \begin{equation}
\label{eqn:cubature_visible}
\beta^{d+h}\;\sum_{\xx \in C \cap \PPrim} f(\xx) e^{-\beta \dotp{\uu }{\xx}} \underset{\beta\downarrow 0}{=} \frac{1}{\zeta(d)}\int_C f(\xx)e^{-\dotp{\uu}{ \xx}}d\xx + O(\beta).
      \end{equation}
\end{proposition}

We start the proof by a standard lemma dealing with the approximation of integrals on a convex body by Riemann sums.

\begin{lemma}
\label{lem:cubature}
Let $L > 0$, and let $K$ be a compact convex subset of the hypercube $[-\frac{L}{2},\frac{L}{2}]^d$. For every Lipschitz continuous function $f \colon K \to \RR$ with Lipschitz constant $M > 0$,
\begin{equation}
\label{eq:cubature}
\left|\sum_{\xx \in K\cap \ZZ^d} f(\xx) - \int_K f(\xx)\,d\xx\right| \leq M\frac{\sqrt{d}}{2} L^d + 4 d! (L+1)^{d-1} \sup_K |f|.
\end{equation}
\end{lemma}

\begin{proof}
For all $\xx$ in $\ZZ^d$, let us consider the (hyper)cube $Q(\xx) = \xx + [-\frac{1}{2},\frac{1}{2}]^d$ of unit volume.
These cubes are of three types: those that are contained in $K$, those that cross the boundary of $K$, and those that have no point in common with $K$.
The idea of the proof is to approximate the integral by considering in first approximation only the reunion of the cubes of the first kind.
For each cube $Q(\xx)$ contained in $K$,
\[
	\left|f(\xx) - \int_{Q(\xx)} f(\mathbf y) d\mathbf y\right| \leq \int_{Q(\xx)} \left|f(\xx) - f(\mathbf y)\right|d\mathbf y \leq M\frac{\sqrt{d}}{2}.
\]
Moreover, the number of cubes contained in $K$ is at most $L^d$.
This already yields the first term in the right-hand side of \eqref{eq:cubature}.

We also have to deal with the cubes that cross the boundary of $K$.
Since $K$ is convex, it is easily seen by induction on $d$ that the number of such cubes is at most $2d!(L+1)^{d-1}$.
Bounding $f$ by $\sup_K |f|$ on $Q(\xx) \cap K$, we obtain \eqref{eq:cubature}.
\end{proof}

\begin{corollary}
\label{cor:cubature_lattice}
Let $f \colon C \to \RR$ be continuously differentiable and positively homogeneous of degree $h$, that is to say $f(\lambda \xx) = \lambda^h f(\xx)$ for all $\xx \in C$ and $\lambda \geq 0$.
Let $A$ be a compact subset of $C^\circ$.
There exists $c_{A,f,d} > 0$ such that for all $\beta \in (0,1]$,
    \begin{equation}
      \label{eqn:cubature_lattice}
        \sup_{\uu \in A}\; \left|\sum_{\xx \in C \cap \ZZ^d} f(\xx) e^{-\beta \dotp{\uu}{ \xx}} - \int_C f(\xx)e^{- \beta \dotp{\uu}{ \xx}}d\xx\right| \leq \frac{c_{A,f,d}}{\beta^{d+h-1}}.
      \end{equation}
\end{corollary}

\begin{proof}
Since $A$ is compact, we can find $L > 0$ such that for all $\uu \in A$, the truncated cone $C_{\uu} = \{\xx \in C \mid \dotp{\uu}{ \xx} \leq 1\}$ is contained in $[-\frac L 2,\frac L 2]^d$.
For all $t > 0$, an application of Lemma~\ref{lem:cubature} to the compact convex subset $t C_{\mathbf u}$ of $[-\frac{tL}{2},\frac{tL}{2}]^d$ implies
\[
    \sup_{\uu \in A}\; \left|\sum_{\xx \in t C_{\uu} \cap \ZZ^d} f(\xx) - \int_{t C_{\uu}} f(\xx) dx \right| \ll_{f,A,d} \left(1+t\right)^{d+h-1}.
\]
By integration over $[0,+\infty)$, we obtain therefore for all $\beta > 0$,
\[
    \sup_{\uu \in A}\; \int_0^\infty \left|  \sum_{\xx \in t C_{\uu} \cap \ZZ^d} f(\xx) - \int_{t C_{ \uu}} f(\xx) dx\right| \beta e^{-\beta t} dt \ll_{f,A,d} \frac{1}{\beta^{d+h-1}}.
\]
This concludes the proof since the Fubini theorem yields
\[
	 \int_0^\infty \left[\sum_{\xx \in C \cap \ZZ^d} f(\xx) \indicator{\dotp{\uu}{ \xx} \leq t}\right] \beta e^{-\beta t}dt =
	\sum_{\xx \in C\cap\ZZ^d} f(\xx) e^{-\beta \dotp{\uu}{ \xx}}
\]
and similarly
\[
	\int_0^\infty \left[\int_{tC_\uu} f(\xx) d\mathbf x\right] \beta e^{-\beta t} dt =  \int_{\xx \in C} f(\xx) e^{-\beta \dotp{\uu}{ \xx}}.
\]
\end{proof}

\begin{proof}[Proof of Proposition~\ref{prop:primitive_sums}]
  This follows from Corollary~\ref{cor:cubature_lattice} and the fact that the subset $\PPrim$ of $\ZZ^d$ has asymptotic density $1/\zeta(d)$, which is a well known consequence of the M\"obius inversion formula \cite{hardy_introduction_2008}.
\end{proof}

Here comes the one-sided inequality needed in the proof of Theorem~\ref{thm:typicalZonotopeVertices}.
\begin{lemma}\label{lem:one_sided}
Let $L>0$, let $H$ be a hyperplane containing the origin.
Let $K$ be a compact convex subset of the hypercube $[-L/2,L/2]^d \cap H$.
For every Lipschitz continuous function $f:K \rightarrow \mathbb{R}$ with Lipschitz constant $M>0$,
$$ \sum_{\mathbf{x} \in K \cap \mathbb{Z}^d} f(\xx) - \int_K f(\mathbf{x}) d\mathbf{x} \leq M \frac{\sqrt{d}}{2}L^{d-1} + 4 (d-1)!(L+1)^{d-2} \sup_K|f|. $$
\end{lemma}
\begin{proof}
For $\mathbf{x} \in \mathbb{Z}^d \cap H$, let $Q(\mathbf{x}) = \mathbf{x} + [-1/2,1/2]^d$ be the hypercube of unit volume centered at $\mathbf{x}$.
A theorem of Vaaler \cite{vaaler1979geometric} establishes that the $(d-1)$-dimensional volume of $H \cap Q(\mathbf{x})$ is at least $1$; let $Q'(\mathbf{x})$ be a subset of $Q(\mathbf{x})$ of volume $1$.

For $Q'(\mathbf{x})$ contained in $K$,
$$\left | f(\mathbf{x}) - \int_{Q'(\mathbf{x})} f(\mathbf{y}) d\mathbf{y} \right | \leq \int_{Q'(\mathbf{x})} |f(\mathbf{x}) - f(\mathbf{y})| \leq M \sqrt{d} / 2.$$

The other error term comes from a bound on the number of cubes that intersect the boundary of $K$, and the proof is exactly as in the proof of Lemma~\ref{lem:cubature}.

We do not have a two-sided inequality as in Lemma~\ref{lem:cubature}, since the union of $Q$ does not contain the intersection $K \cap H$, and we have also lost some of the mass by passing from $Q$ to $Q'$.
\end{proof}

\section*{Acknowledgments}

The authors are pleased to thank Nathana\"el Enriquez for his insightful suggestions about this work.
This research was supported by ERC Advanced Research Grant no 267165
(DISCONV). Imre B\'ar\'any is partially supported by Hungarian National
Research Grant K 111827.
Ben Lund is partially supported by NSF grant CCF-1350572.

\bibliography{bibexport}

\end{document}